\documentclass[leqno]{amsart}

\usepackage[T1]{fontenc}
\usepackage[utf8]{inputenc}
\usepackage[english]{babel}
\usepackage[a4paper,vmargin={3cm,3cm},hmargin={3cm,3cm}]{geometry}
\usepackage{amsmath,amssymb,amsthm,mathtools}
\usepackage{microtype}
\usepackage{enumitem}
\usepackage[colorlinks=true,allcolors=blue]{hyperref}
\usepackage{cite}

\newtheorem{theorem}{Theorem}[section]
\newtheorem{conjecture}[theorem]{Conjecture}
\newtheorem{proposition}[theorem]{Proposition}
\newtheorem{lemma}[theorem]{Lemma}
\newtheorem{corollary}[theorem]{Corollary}

\theoremstyle{remark}
\newtheorem{remark}[theorem]{Remark}
\newtheorem{example}[theorem]{Example}

\newcommand{\N}{\mathbb N}
\newcommand{\R}{\mathbb R}
\newcommand{\Z}{\mathbb Z}
\newcommand{\C}{\mathbb C}
\newcommand{\e}{\mathrm e}
\newcommand{\norm}[1]{\lVert #1\rVert}
\newcommand{\abs}[1]{\lvert #1\rvert}
\newcommand{\dist}[1]{\left\lVert #1\right\rVert_{\R/\Z}}
\renewcommand{\le}{\leqslant}
\renewcommand{\ge}{\geqslant}

\numberwithin{equation}{section}
\allowdisplaybreaks[1]
\title[Erd\H{o}s--Wintner theorem for second-order bases]
{An Erd\H{o}s--Wintner theorem for second-order linear recurrent bases}
\author{Johann Verwee}

\begin{document}

\begin{abstract}
Let $a,b$ be integers with $1\le b\le a$, and let
\[
  G_0=1,\qquad G_1=a+1,\qquad
  G_{n+2}=aG_{n+1}+bG_n.
\]
For real-valued functions which are additive in the greedy $G$-digits, we prove
a necessary-and-sufficient criterion for the existence of a limiting
distribution.  The criterion consists of a first-order drift series and a
quadratic digit-energy series, and it recovers the Zeckendorf theorem when
$a=b=1$.  The main difficulty is necessity: a one-step transfer matrix detects
all non-maximal digit values, while the maximal digit becomes visible only in a
two-step product.  A weighted Euclidean norm symmetrizes the untwisted
companion matrix and makes both contractions occur at the true Perron scale.
Sufficiency follows from a two-dimensional Perron product lemma with
square-summable transverse perturbations and a convergent, not necessarily
absolutely convergent, Perron drift.
The limiting characteristic function admits a scalar infinite-product
representation in a neighbourhood of the origin and a global matrix-product
representation at arbitrary frequencies.
\end{abstract}

\maketitle
\tableofcontents

\section{Introduction}

The classical theorem of Erd\H{o}s and Wintner characterizes the additive
arithmetical functions which possess a limiting distribution
\cite{ErdosWintner}. Tenenbaum and the present author subsequently obtained
an effective form of this theorem, giving quantitative estimates for the
convergence to the limiting distribution \cite{TenenbaumVerwee}. In the
digital setting, Delange proved the analogous necessary-and-sufficient
criterion for $q$-additive functions \cite{Delange}. Quantitative versions,
together with extensions to Cantor and Zeckendorf numeration, were later
obtained by Drmota and the present author \cite{DrmotaVerwee}. More recently,
the present author removed the boundedness assumption on the Cantor base and
obtained an effective Erd\H{o}s--Wintner theorem in this more general setting
\cite{VerweeCantor}.

For canonical numeration systems attached to linear recurrences with
decreasing coefficients, Barat and Grabner established broad sufficient
conditions and developed the transfer-recurrence approach used here
\cite{BaratGrabner}. Their proof uses a Poincar\'e--Perron perturbation theorem
of Kooman whose coefficient hypothesis is absolute summability; the remark
following their Theorem~4 identifies this perturbative step as the obstruction
to obtaining necessity by the same argument. In the Zeckendorf case, Drmota
and the present author nevertheless obtained a complete necessary-and-sufficient
criterion \cite{DrmotaVerwee}. In the closing remark of that paper, the
extension to constant-coefficient recurrent bases was left open, with the
recurrence shape $G_{n+2}=aG_{n+1}+G_n$ singled out as a natural next case.

An earlier preprint by the present author considered this problem in greater
generality \cite{VerweeLRBWithdrawn} and was subsequently withdrawn. The present
paper is self-contained and establishes the second-order result independently.

The present paper treats the second-order family
\[
  G_{n+2}=aG_{n+1}+bG_n,\qquad 1\le b\le a.
\]
The point is not to obtain another sufficient criterion. Such criteria are
already available in substantially greater generality. The new issue is the
converse: starting only from the existence of a limiting distribution, we
recover both the quadratic digit energy and the first-order drift. This
extends the necessary part of the Zeckendorf theorem from $(a,b)=(1,1)$ to an
entire two-parameter family.

The key new feature is that the transfer recurrence treats the maximal digit
differently from the non-maximal ones: the non-maximal digits can be controlled
at one step, whereas the maximal digit requires grouping two consecutive
transfer matrices.  A weighted norm adapted to the Perron direction makes
these estimates sharp enough to recover the quadratic condition, while the
first-order part of the same recurrence yields the drift condition.  When the
criterion holds, the limiting characteristic function has a scalar
infinite-product representation near the origin; at arbitrary frequencies it
is represented instead by the normalized transfer-matrix product.

The paper is organized as follows. Section~\hyperlink{sec:numeration}{2}
introduces the second-order numeration system, its greedy language, and the exact
transfer recurrence. Section~\hyperlink{sec:main}{3} states and proves the main
result, following the mechanism of the argument: sufficiency, passage from
natural cutoffs to arbitrary intervals, infinitesimality, the quadratic
condition, and finally the drift condition. Section~\hyperlink{sec:examples}{4}
includes an example with conditionally convergent drift, while
Section~\hyperlink{sec:special-cases}{5} records several specializations and two
immediate stability consequences. Section~\hyperlink{sec:further-directions}{6}
formulates several directions beyond order two. The proof of the perturbative
product lemma is given in Appendix~\hyperlink{app:perron-product}{A}.

\hypertarget{sec:numeration}{}
\section{Second-order recurrent numeration systems}

Fix integers $a,b$ with $1\le b\le a$, and put
\[
  G_0=1,\qquad G_1=a+1,\qquad
  G_{n+2}=aG_{n+1}+bG_n\qquad(n\ge0).
\]
Let
\[
  \alpha=\frac{a+\sqrt{a^2+4b}}2,
  \qquad
  \lambda=\frac{a-\sqrt{a^2+4b}}2=-\frac b\alpha.
\]
Then $\alpha>a$, $\abs{\lambda}<1$, and
\begin{equation}\label{eq:G-asymp}
  G_n=\kappa\alpha^n+(1-\kappa)\lambda^n,
  \qquad
  \kappa=\frac{a+2+\sqrt{a^2+4b}}{2\sqrt{a^2+4b}}>0.
\end{equation}

For an integer $N\ge1$, a digit word $(e_0,\ldots,e_{N-1})$ is called \emph{admissible} if
\begin{equation}\label{eq:admissibility}
  0\le e_j\le a,
  \qquad
  e_{j+1}=a\ \Longrightarrow\ e_j<b.
\end{equation}

\begin{proposition}\label{prop:numeration}
For every $N\ge1$, the value map
\[
  (e_0,\ldots,e_{N-1})\longmapsto\sum_{j<N}e_jG_j
\]
is a bijection from the admissible words of length $N$ onto
$\{0,\ldots,G_N-1\}$.  Consequently every non-negative integer has a unique
greedy expansion
\[
  n=\sum_{j\ge0}e_j(n)G_j
\]
with finitely many non-zero digits satisfying \eqref{eq:admissibility}.
\end{proposition}

\begin{proof}
We argue by induction on the length, allowing the empty word at length zero.
The cases $N=0$ and $N=1$ are immediate.  Suppose the assertion is known at
lengths $N$ and $N+1$.  For an admissible word of length $N+2$, either its
leading digit $e_{N+1}$ is some $k<a$, in which case the lower $N+1$ digits
are unrestricted and give exactly
\[
  kG_{N+1}+[0,G_{N+1}),
\]
or $e_{N+1}=a$, in which case admissibility forces $e_N=k<b$ and the lower
$N$ digits give exactly
\[
  aG_{N+1}+kG_N+[0,G_N).
\]
As $k$ ranges over the allowed values, these intervals are consecutive and
disjoint, and their union is
\[
  [0,aG_{N+1}+bG_N)=[0,G_{N+2}).
\]
This proves the bijection at length $N+2$.  Letting the length increase gives
a unique finite admissible expansion of every non-negative integer.  Every
prefix of this expansion has value below the next place value, again by the
bijection just proved, so the expansion is the greedy one.
\end{proof}

This is the order-two specialization of the canonical recurrent numeration
systems of Grabner--Tichy and Barat--Grabner; see
\cite{GrabnerTichy} and, in particular, the lexicographic admissibility
criterion \cite[Eq.~(1.3)]{BaratGrabner}.  In the present case that criterion
reduces exactly to \eqref{eq:admissibility}.

A function $f:\N\to\R$ is called \emph{$G$-additive} if $f(0)=0$ and,
for every $n\ge1$,
\[
  f(n)=\sum_{j\ge0}f(e_j(n)G_j).
\]
Similarly, a function
$g:\N\to\C$ is called \emph{$G$-multiplicative} if $g(0)=1$ and
\[
  g(n)=\prod_{j\ge0}g(e_j(n)G_j).
\]
For $N\ge1$ and $x\in\R$, let
\[
  F_N(x):=\frac1N\#\{0\le n<N:f(n)\le x\}.
\]
We say that a $G$-additive function $f$ has a \emph{limiting distribution} if there exists a distribution
function $F$ such that $F_N(x)\to F(x)$ at every continuity point of $F$.
Equivalently, the empirical probability measures on $[0,N)$ converge weakly.
For $t\in\R$, set
\[
  g_t(n)=\e^{itf(n)},
  \qquad
  H_n(t)=\sum_{m<G_n}g_t(m).
\]
For $c\in\{a,b\}$, define
\[
  \sigma_n(c;t)=\sum_{0\le k<c}\e^{itf(kG_n)}.
\]
The dependence on $t$ will be omitted when no confusion is possible.

The admissibility rule gives the disjoint decomposition
\begin{equation}\label{eq:block-decomp}
[0,G_{n+2})=
  \bigsqcup_{0\le k<a}\bigl(kG_{n+1}+[0,G_{n+1})\bigr)\sqcup
  \bigsqcup_{0\le k<b}
  \bigl(aG_{n+1}+kG_n+[0,G_n)\bigr).
\end{equation}
Indeed, a leading digit smaller than $a$ leaves the lower positions
unrestricted, whereas a leading digit equal to $a$ forces the following digit
to lie in $\{0,\ldots,b-1\}$.  Since $g_t$ is $G$-multiplicative, \eqref{eq:block-decomp} yields
\begin{equation}\label{eq:H-rec}
  H_{n+2}(t)
  =\sigma_{n+1}(a;t)H_{n+1}(t)
   +g_t(aG_{n+1})\sigma_n(b;t)H_n(t).
\end{equation}
This is the order-two specialization of the general transfer recurrence
in \cite[Eq.~(2.6)]{BaratGrabner}. Equivalently, for $n\ge1$,
\begin{equation}\label{eq:matrix-rec}
  \binom{H_{n+1}(t)}{H_n(t)}
  =A_n(t)\binom{H_n(t)}{H_{n-1}(t)},
  \qquad
  A_n(t)=
  \begin{pmatrix}
    \sigma_n(a;t)&g_t(aG_{n})\sigma_{n-1}(b;t)\\
    1&0
  \end{pmatrix}.
\end{equation}
At $t=0$ this matrix is
\[
  A=\begin{pmatrix}a&b\\1&0\end{pmatrix}.
\]
A right and a left Perron eigenvector are
\begin{equation}\label{eq:Perron-vectors}
  r=\binom{\alpha}{1},
  \qquad
  \ell^{\mathsf T}=\begin{pmatrix}1&b/\alpha\end{pmatrix}.
\end{equation}

\hypertarget{sec:main}{}
\section{Main theorem}

The quantity $D_n$ is the first-order drift of the $n$-th transfer matrix
in the Perron direction.  We put
\begin{equation}\label{eq:Dn}
  D_n
  =\alpha\sum_{1\le k<a}f(kG_n)
   +\sum_{1\le k<b}f(kG_{n-1})
   +b f(aG_n)
  \qquad(n\ge1).
\end{equation}
The omission of the zero digit is harmless because $f(0)=0$.

\begin{theorem}\label{thm:main}
Let $f$ be a real-valued $G$-additive function. Then $f$ has a limiting
distribution if and only if the two series
\begin{align}
  &\sum_{n\ge1}D_n,\tag{S1}\label{eq:S1}\\
  &\sum_{n\ge0}\sum_{1\le k\le a}f(kG_n)^2\tag{S2}\label{eq:S2}
\end{align}
converge.

In this case, if $\Phi$ denotes the characteristic function of the limiting
distribution, then there exists $t_0>0$ such that $H_n(t)\ne0$ for every
$n\ge0$ and $\abs t\le t_0$.  For such $t$,
\begin{equation}\label{eq:ratio-product-Phi}
  \Phi(t)
  =\frac1\kappa\prod_{n\ge1}\frac{r_n(t)}{\alpha},
  \qquad
  r_n(t)=\frac{H_n(t)}{H_{n-1}(t)}.
\end{equation}
\end{theorem}

The scalar infinite-product representation in Theorem~\ref{thm:main} is local.
At arbitrary frequencies the ratios need not be defined, since the limiting
characteristic function may vanish.  Under the equivalent conditions of
Theorem~\ref{thm:main}, the corresponding global representation is
\begin{equation}\label{eq:matrix-product-Phi}
  \Phi(t)
  =\frac1\kappa\lim_{n\to\infty}
   \alpha^{-n}
   \begin{pmatrix}1&0\end{pmatrix}
   A_{n-1}(t)\cdots A_1(t)
   \binom{H_1(t)}{1}
  \qquad(t\in\R).
\end{equation}
Thus the scalar infinite product in \eqref{eq:ratio-product-Phi} is a local
formula near the origin, while \eqref{eq:matrix-product-Phi} remains valid
globally and does not require division by the finite characteristic sums.

Barat and Grabner's Theorem~4 gives a sufficient criterion for the more
general decreasing-coefficient recurrence setting
\cite[Theorem~4]{BaratGrabner}.  Their perturbative step, however, uses
absolute summability of the coefficient perturbations, and they explicitly
note this limitation in the remark following that theorem.  The condition
\eqref{eq:S1} used here allows the drift in the Perron direction to converge
conditionally, so that this regime is not covered in general by that argument.

The proof has four steps, and the two canonical series enter the argument in
different ways.

\begin{enumerate}[label=\textup{(\roman*)}]
\item Under \eqref{eq:S2}, the transfer matrices are an $\ell^2$ perturbation
of the constant companion matrix.  Under \eqref{eq:S1}, the corresponding
Perron drift series converges.  Proposition~\ref{prop:product} then gives convergence at the natural
cutoffs $G_n$.
\item A greedy block decomposition passes from the cutoffs $G_n$ to every
initial interval $[0,N)$.
\item Conversely, local non-vanishing of $\Phi$ near the origin first forces
all digit values $f(kG_n)$ to tend to zero.  The weighted one-step and two-step
contractions then yield \eqref{eq:S2}.
\item Once \eqref{eq:S2} is known, the ratios are used on the same
non-vanishing interval to recover the summability of the Perron drift and hence
\eqref{eq:S1}.
\end{enumerate}

\begin{remark}
With the change of index $n=j+1$, the drift condition \eqref{eq:S1} is exactly
\[
  \sum_{j\ge0}\left(
    \alpha\sum_{1\le k<a}f(kG_{j+1})
    +\sum_{1\le k<b}f(kG_j)
    +b f(aG_{j+1})
  \right),
\]
which is the second-order drift appearing naturally in the linearization of
the recurrence.  Thus no averaging factor or boundary correction is hidden in
the notation $D_n$.
\end{remark}

\subsection{Sufficiency}\label{sec:suff}

We first prove the implication from \eqref{eq:S1} and \eqref{eq:S2} to the
existence of a limiting distribution.  The argument has two steps: first the
natural cutoffs $G_n$, and then arbitrary initial intervals.

We begin by recording the perturbative result used in the proof.  Its formulation is
specialized to the two-dimensional situation needed here.  General
asymptotic-integration theories for perturbed difference systems go back, in
particular, to Coffman and to Benzaid--Lutz
\cite{Coffman,BenzaidLutz}.  Barat--Grabner instead invoke a theorem of Kooman
under a stronger absolute-summability hypothesis
\cite[Lemma~2]{BaratGrabner}.  Since we need an $\ell^2$ perturbation together
with convergence, but not necessarily absolute convergence, in the Perron
direction, we keep the precise order-two statement and its proof.

\begin{proposition}\label{prop:product}
Let $A\in\C^{2\times2}$ be diagonalizable with eigenvalues $\alpha$ and
$\lambda$, where $\alpha>0$ and $\abs{\lambda}<\alpha$.  Let $r$ and $\ell$ be
right and left $\alpha$-eigenvectors with $\ell^{\mathsf T}r\ne0$, and fix any
matrix norm on $\C^{2\times2}$.  Suppose that $A_n=A+E_n$ and
\[
  \sum_{n\ge0}\norm{E_n}^2<\infty,
  \qquad
  \sum_{n\ge0}\ell^{\mathsf T}E_n r\quad\text{converges}.
\]
Then the normalized products
\[
  \alpha^{-N}A_{N-1}\cdots A_0
\]
converge in $\C^{2\times2}$ as $N\to\infty$.

Suppose now that $A_n=A_n(t)$ depends continuously on $t$, and write
$E_n(t)=A_n(t)-A$.  If, for every $T>0$,
\[
  \sum_{n\ge0}\sup_{|t|\le T}\norm{E_n(t)}^2<\infty
\]
and
\[
  \sum_{n\ge0}\ell^{\mathsf T}E_n(t)r
\]
converges uniformly for $|t|\le T$, then the normalized products converge
locally uniformly in $t$, and the limiting matrix is continuous.
\end{proposition}

The square-summability hypothesis controls the quadratic interactions of the
perturbation, while $\ell^{\mathsf T}E_n r$ records its first-order drift
along the Perron line.  The
proof is given in the appendix; keeping it separate allows the
application to the digit matrices to remain transparent.

For $n\ge1$, set
\begin{align*}
  p_n(t)&=\sigma_n(a;t)-a,\\
  q_n(t)&=g_t(aG_n)\sigma_{n-1}(b;t)-b,\\
  u_n(t)&=\alpha p_n(t)+q_n(t)
          =\ell^{\mathsf T}(A_n(t)-A)r.
\end{align*}

\begin{lemma}\label{lem:taylor-u}
Assume \eqref{eq:S2}.  For every $T>0$,
\begin{align}
  \sum_{n\ge1}&\sup_{\abs t\le T}\norm{A_n(t)-A}^2<\infty,\notag\\
  u_n(t)&=itD_n+O_{a,b}\!\left(
    T^2\sum_{1\le k\le a}
      \bigl(f(kG_n)^2+f(kG_{n-1})^2\bigr)
  \right).\label{eq:u-expansion}
\end{align}
Consequently, if the series \eqref{eq:S1} also converges, then
$\sum_n u_n(t)$ converges locally uniformly in $t$.
\end{lemma}

\begin{proof}
The basic second-order exponential estimate is the same one used in
\cite[Lemma~1]{BaratGrabner}.  In the present notation, we use the uniform
inequalities
\[
  \abs{\e^{ix}-1}\le\abs{x},
  \qquad
  \abs{\e^{ix}-1-ix}\le\frac{x^2}{2}
  \qquad(x\in\R).
\]
The upper-left entry of $A_n(t)-A$ is a sum of terms of the form
$\e^{itf(kG_n)}-1$, while the upper-right entry is a sum of terms
\[
  \e^{it(f(aG_n)+f(kG_{n-1}))}-1\qquad(0\le k<b).
\]
Using $|x+y|^2\le2x^2+2y^2$, the first exponential inequality above, and
the equivalence of norms on $\C^{2\times2}$, condition \eqref{eq:S2}
therefore implies the first assertion.

For the linear term, we have
\[
  p_n(t)=it\sum_{1\le k<a}f(kG_n)
  +O_a\left(T^2\sum_{1\le k<a}f(kG_n)^2\right)
\]
uniformly for $|t|\le T$.  Similarly,
\[
  q_n(t)=it\left(
    \sum_{1\le k<b}f(kG_{n-1})+b f(aG_n)
  \right)
  +O_{a,b}\left(T^2\sum_{1\le k\le a}
    \bigl(f(kG_n)^2+f(kG_{n-1})^2\bigr)\right).
\]
Combining these expansions in $u_n=\alpha p_n+q_n$ gives
\eqref{eq:u-expansion}.  If \eqref{eq:S1} converges, its linear contribution
is locally uniformly convergent in $t$, while the displayed quadratic
remainder is locally uniformly absolutely summable.
\end{proof}

We now apply these estimates at the natural cutoffs.  Assume \eqref{eq:S1}
and \eqref{eq:S2}.  The quadratic condition implies
$f(kG_n)\to0$ for every $1\le k\le a$.  For the Perron vectors
\eqref{eq:Perron-vectors},
\[
  \ell^{\mathsf T}(A_n(t)-A)r=u_n(t).
\]
Lemma~\ref{lem:taylor-u} gives the required locally uniform $\ell^2$ estimate.
Its expansion of $u_n(t)$ shows that $\sum_n u_n(t)$ converges locally uniformly:
the linear part is $it\sum_n D_n$, while the quadratic remainder is locally
uniformly absolutely summable.  Proposition~\ref{prop:product} therefore gives
a locally uniform limit of $H_n(t)/\alpha^n$.  By \eqref{eq:G-asymp}, we have $G_n/\alpha^n\to\kappa>0$. Hence there is a continuous function $\Phi$, with $\Phi(0)=1$, such that
\begin{equation}\label{eq:natural-cutoff-limit}
  \frac{H_n(t)}{G_n}\longrightarrow\Phi(t)
\end{equation}
locally uniformly in $t$.

The preceding argument only treats the canonical intervals $[0,G_n)$.  We
now pass to arbitrary initial intervals.  The following lemma is the form of
\cite[Lemma~3]{BaratGrabner} needed here, with the infinitesimality condition
on the digit factors made explicit when the limiting value is non-zero.  We
include the short proof in the present order-two setting.

\begin{lemma}\label{lem:all-cutoffs}
Let $g$ be a $G$-multiplicative function with $\abs g\le1$, and suppose that
\[
  \frac1{G_n}\sum_{m<G_n}g(m)\longrightarrow L.
\]
Assume in addition that either $L=0$, or
\[
  g(kG_n)\longrightarrow1
  \qquad(1\le k\le a).
\]
Then
\[
  \frac1N\sum_{m<N}g(m)\longrightarrow L.
\]
\end{lemma}

\begin{proof}
Write
\[
  S(M):=\sum_{m<M}g(m),
\]
and let
\[
  N=\sum_{q=0}^{J}e_qG_q,
  \qquad e_J>0,
\]
be the greedy expansion of $N$.  Every $m<N$ has a unique highest
position $q$ at which its greedy expansion differs from that of $N$; the
digits above $q$ agree, while its $q$-th digit is strictly smaller than
$e_q$.  This gives a disjoint decomposition of $[0,N)$ into blocks obtained
by fixing the digits above $q$, choosing the $q$-th digit strictly smaller
than $e_q$, and leaving the lower $q$ digits free.  If $e_{q+1}=a$, then
admissibility gives $e_q<b$, so every choice $c<e_q$ is still compatible with
the digit above; and since $c<a$, there is no restriction on the lower digits.
Proposition~\ref{prop:numeration} therefore gives
\[
  S(N)=\sum_{q=0}^{J}C_q(N)S(G_q),
\]
where
\[
  C_q(N):=
  \left(\prod_{q<r\le J}g(e_rG_r)\right)
  \sum_{0\le c<e_q}g(cG_q).
\]
In particular,
\[
  \abs{C_q(N)}\le e_q,
  \qquad
  N=\sum_{q=0}^{J}e_qG_q.
\]

Put
\[
  \varepsilon_q:=\frac{S(G_q)}{G_q}-L.
\]
Then $\varepsilon_q\to0$, and
\[
  \frac{S(N)}N-L
  =\frac1N\sum_{q=0}^{J}C_q(N)G_q\varepsilon_q
   +\frac{L}{N}\sum_{q=0}^{J}\bigl(C_q(N)-e_q\bigr)G_q.
\]
The first term tends to zero.  Indeed, after fixing $Q$, the contribution of
$q<Q$ is $O_Q(N^{-1})$, while the remaining contribution is bounded by
$\sup_{q\ge Q}\abs{\varepsilon_q}$ because
$\sum_q e_qG_q=N$.

If $L=0$, this already proves the result.  Assume now that
$g(kG_n)\to1$ for $1\le k\le a$.  Since $g(0)=1$ and there are only finitely
many digits, we have
\[
  \delta_n:=\max_{0\le k\le a}\abs{g(kG_n)-1}\longrightarrow0.
\]
It remains to show that
\[
  \frac1N\sum_{q=0}^{J}\bigl(C_q(N)-e_q\bigr)G_q\longrightarrow0.
\]
Fix an integer $h\ge1$.  Since $\abs{C_q(N)-e_q}\le2a$ and
$G_J\le N$, formula~\eqref{eq:G-asymp} gives
\[
  \frac1N\sum_{q\le J-h}
  \abs{C_q(N)-e_q}G_q
  \ll \frac{G_{J-h}}{G_J}
  \ll \alpha^{-h}.
\]
For the remaining $h$ positions, put
\[
  \Delta_{J,h}:=
  \max_{J-h<r\le J}\delta_r.
\]
For fixed $h$, we have $\Delta_{J,h}\to0$ as $J\to\infty$.  Moreover, if
$J-h<q\le J$, the product defining $C_q(N)$ contains at most $h$ factors,
and the elementary telescoping estimate for products of numbers of modulus at
most one yields
\[
  \abs{C_q(N)-e_q}\le a(h+1)\Delta_{J,h}.
\]
Hence, again by \eqref{eq:G-asymp},
\[
  \frac1N\sum_{J-h<q\le J}
  \abs{C_q(N)-e_q}G_q=o_h(1).
\]
We conclude that
\[
  \limsup_{N\to\infty}
  \frac1N\left|\sum_{q=0}^{J}\bigl(C_q(N)-e_q\bigr)G_q\right|
  \ll \alpha^{-h}.
\]
Letting $h\to\infty$ proves the claim.
\end{proof}

We can now complete the sufficiency direction of Theorem~\ref{thm:main}.
Apply Lemma~\ref{lem:all-cutoffs} to
$g_t(n)=\e^{itf(n)}$.  Equation~\eqref{eq:natural-cutoff-limit} supplies the
limit at the natural cutoffs.  Moreover, \eqref{eq:S2} implies
$f(kG_n)\to0$ for every $1\le k\le a$, and hence
\[
  g_t(kG_n)=\e^{itf(kG_n)}\longrightarrow1.
\]
Thus the characteristic functions of the empirical measures on $[0,N)$
converge pointwise to $\Phi$.  Since $\Phi$ is continuous at the origin and
$\Phi(0)=1$, L\'evy's continuity theorem gives a limiting distribution.

\subsection{Necessity}

We now prove the converse implication.  Assume throughout this subsection that
$f$ has a limiting distribution with characteristic function $\Phi$.

At the natural cutoffs this gives
\begin{equation}\label{eq:H-to-Phi}
  \frac{H_n(t)}{G_n}\longrightarrow\Phi(t)
  \qquad(t\in\R).
\end{equation}
Since $\Phi(0)=1$, there exists $t_0>0$ such that
\begin{equation}\label{eq:Phi-nonzero}
  \abs{\Phi(t)}\ge\frac12
  \qquad(\abs t\le t_0).
\end{equation}
All subsequent ratio arguments are restricted to this interval; the matrix
arguments do not require any global non-vanishing statement.

\begin{lemma}\label{lem:infinitesimal}
For every $1\le k\le a$,
\[
  f(kG_n)\longrightarrow0.
\]
\end{lemma}

\begin{proof}
Fix $|t|\le t_0$ and write $\phi_n(t)=H_n(t)/G_n$.  With
\[
  \tau_n(t)=g_t(aG_{n+1})\sigma_n(b;t),
\]
recurrence \eqref{eq:H-rec} gives
\begin{align*}
  \phi_{n+2}-\Phi={}&
  \sigma_{n+1}(a;t)\frac{G_{n+1}}{G_{n+2}}
     (\phi_{n+1}-\Phi)\\
  &+\tau_n(t)\frac{G_n}{G_{n+2}}(\phi_n-\Phi)\\
  &+\Phi\left(
    \sigma_{n+1}(a;t)\frac{G_{n+1}}{G_{n+2}}
    +\tau_n(t)\frac{G_n}{G_{n+2}}-1\right).
\end{align*}
The first two terms tend to zero because the digit sums are uniformly bounded.
Using $G_{n+1}/G_{n+2}\to\alpha^{-1}$ and
$G_n/G_{n+2}\to\alpha^{-2}$, and recalling that $\Phi(t)\ne0$, we obtain
\[
  \alpha\sigma_{n+1}(a;t)+\tau_n(t)-\alpha^2\longrightarrow0.
\]
This is precisely $u_{n+1}(t)\to0$.  Its real part is
\begin{align*}
  \operatorname{Re}u_n(t)
  ={}&-\alpha\sum_{0\le k<a}
       \bigl(1-\cos(tf(kG_n))\bigr)\\
     &-\sum_{0\le k<b}
       \bigl(1-\cos(t(f(aG_n)+f(kG_{n-1})))\bigr).
\end{align*}
Every summand is non-positive.  Therefore, for every $\abs t\le t_0$,
\[
  \cos(tf(kG_n))\to1\quad(1\le k<a),
  \qquad
  \cos(tf(aG_n))\to1,
\]
the second assertion following from the term $k=0$ in the second sum.

It remains to remove the possible multiples of $2\pi$.  We use the elementary
implication
\[
  \cos(tx_n)\longrightarrow1\quad\text{for every $t$ in a non-trivial interval}
  \quad\Longrightarrow\quad x_n\longrightarrow0.
\]
Indeed, by dominated convergence,
\[
  \int_{-t_0}^{t_0}\bigl(1-\cos(tx_n)\bigr)\,dt
  =2t_0-\frac{2\sin(t_0x_n)}{x_n}\longrightarrow0,
\]
with the quotient interpreted by continuity at $x_n=0$.  For every
$\varepsilon>0$, the expression on the right is bounded below by a positive
constant whenever $|x_n|\ge\varepsilon$: it is continuous and strictly
positive away from zero, and tends to $2t_0$ as $|x_n|\to\infty$.  Hence
$x_n\to0$.  Applying this to each digit-value sequence proves the lemma.
\end{proof}

We next prove the quadratic condition \eqref{eq:S2}.  The two contractions have complementary roles.  In the Zeckendorf case, the
necessity proof in \cite[proof of Theorem~7]{DrmotaVerwee} already uses a
spectral-norm deficit after grouping two consecutive transfer matrices.  Here
a single matrix contains the phases of the digits $1,\ldots,a-1$ in its first
entry, while the maximal digit $a$ occurs in an entry whose modulus is
unchanged and cannot produce a one-step norm deficit.  In a product of two
consecutive matrices, that phase interferes with an untwisted term and becomes
detectable.  The next two lemmas are the order-two estimates needed to separate
these two effects.

Put
\[
  W=\begin{pmatrix}\sqrt b&0\\0&1\end{pmatrix},
  \qquad
  B_n(t)=W^{-1}A_n(t)W,
  \qquad
  C=W^{-1}AW=
  \begin{pmatrix}a&\sqrt b\\\sqrt b&0\end{pmatrix}.
\]
The matrix $C$ is symmetric and $\norm C_2=\alpha$.

The elementary phase-cancellation mechanism in the next estimate is the same
as in the $q$-adic bound used in the proof of the classical digital
Erd\H{o}s--Wintner theorem; compare \cite[Eq.~(2.6)]{DrmotaVerwee}.  What is
specific here is the weighted Perron norm dictated by the companion matrix.

\begin{lemma}\label{lem:one-step}
There exists $c_1>0$, depending only on $a,b$, such that
\[
  \norm{B_n(t)}_2
  \le\alpha\exp\!\left(
    -c_1\sum_{1\le k<a}
       \dist{\frac{tf(kG_n)}{2\pi}}^2
  \right).
\]
\end{lemma}

\begin{proof}
For every complex matrix $M$, the componentwise inequality
$|Mx|\le |M|\,|x|$ gives $\|M\|_2\le\||M|\|_2$.  Here, entrywise,
\[
  |B_n(t)|
  \le
  \begin{pmatrix}|\sigma_n(a;t)|&\sqrt b\\\sqrt b&0\end{pmatrix}.
\]
The norm of the non-negative symmetric matrix on the right is
\[
  h(x)=\frac{x+\sqrt{x^2+4b}}2
  \quad\text{at }x=\abs{\sigma_n(a;t)}.
\]
Since $h(a)=\alpha$ and $h'$ is bounded below by a positive constant on
$[0,a]$, one has
\[
  \alpha-h(x)\gg_{a,b}a-x
  \geqslant \frac{a^2-x^2}{2a}.
\]
Moreover, using the elementary bound
$1-\cos x\ge 8\dist{x/(2\pi)}^2$ for every $x\in\mathbb R$, we obtain
\[
  a^2-\abs{\sigma_n(a;t)}^2
  \ge2\sum_{1\le k<a}(1-\cos(tf(kG_n)))
  \ge16\sum_{1\le k<a}
      \dist{\frac{tf(kG_n)}{2\pi}}^2.
\]
Thus $\norm{B_n(t)}_2\le\alpha-c\mathcal E_n(t)$, where
\[
  \mathcal E_n(t)=\sum_{1\le k<a}\dist{\frac{tf(kG_n)}{2\pi}}^2.
\]
After decreasing $c$ so that $0\le c\mathcal E_n(t)/\alpha\le1$, the inequality
$1-y\le\e^{-y}$ proves the claimed estimate.
\end{proof}

A single step need not detect the contribution of the maximal digit, so we now
group two consecutive matrices.

\begin{lemma}\label{lem:two-step}
There exists $c_2>0$, depending only on $a,b$, such that
\[
  \norm{B_{n+1}(t)B_n(t)}_2
  \le\alpha^2\exp\!\left(
    -c_2\dist{\frac{tf(aG_{n+1})}{2\pi}}^2
  \right).
\]
\end{lemma}

\begin{proof}
The upper-left entry of $B_{n+1}B_n$ is
\[
  S_n(t)=\sigma_{n+1}(a;t)\sigma_n(a;t)
    +g_t(aG_{n+1})\sigma_n(b;t).
\]
This is a sum of $m=a^2+b$ complex numbers of modulus one.  Among them are a
term of phase $0$ and a term of phase $tf(aG_{n+1})$.  Hence
\[
  m^2-\abs{S_n(t)}^2
  \ge2\bigl(1-\cos(tf(aG_{n+1}))\bigr),
\]
and therefore
\begin{equation}\label{eq:delta}
  \delta_n:=m-\abs{S_n(t)}
  \gg_{a,b}\dist{\frac{tf(aG_{n+1})}{2\pi}}^2.
\end{equation}
The remaining entries of $\abs{B_{n+1}B_n}$ are bounded by the corresponding
entries of
\[
  C^2=\begin{pmatrix}m&a\sqrt b\\a\sqrt b&b\end{pmatrix}.
\]
Thus, entrywise,
\[
  \abs{B_{n+1}B_n}
  \le C^2-\delta_n
          \begin{pmatrix}1&0\\0&0\end{pmatrix}.
\]
The matrix on the right is symmetric and entrywise non-negative.  Monotonicity
of the Euclidean operator norm under entrywise domination by a non-negative
matrix therefore gives
\[
  \norm{B_{n+1}B_n}_2
  \le\rho\!\left(C^2-\delta_n
          \begin{pmatrix}1&0\\0&0\end{pmatrix}\right).
\]
By symmetry and the Perron--Frobenius theorem, this spectral radius is also
the Euclidean operator norm.  Writing
\[
  \Lambda(\delta)=\frac{m-\delta+b+
  \sqrt{(m-\delta-b)^2+4a^2b}}2,
\]
we have $\Lambda(0)=\alpha^2$ and
\[
  -\Lambda'(\delta)
  =\frac12\left(1+\frac{m-\delta-b}
   {\sqrt{(m-\delta-b)^2+4a^2b}}\right)>0.
\]
The last expression has a positive minimum on $[0,m]$.  Hence
$\Lambda(\delta)\le\alpha^2-c\delta$ there.  Combining this with \eqref{eq:delta} gives
\[
  \norm{B_{n+1}(t)B_n(t)}_2
  \le \alpha^2-c\dist{\frac{tf(aG_{n+1})}{2\pi}}^2.
\]
After decreasing $c$ so that the normalized deficit is at most $1$, the
inequality $1-y\le\e^{-y}$ proves the claimed estimate.
\end{proof}

We can now combine the preceding contraction estimates to recover
\eqref{eq:S2}.

\begin{proposition}\label{prop:S2-necessity}
If $f$ has a limiting distribution, then the quadratic series
\eqref{eq:S2} converges.
\end{proposition}

\begin{proof}
By Lemma~\ref{lem:infinitesimal}, the finitely many sequences
$(f(kG_n))_n$, $1\le k\le a$, are bounded.  Choose once and for all a non-zero
$t$ with $|t|\le t_0$ so small that
\begin{equation}\label{eq:small-fixed-t}
  |tf(kG_n)|\le\pi
  \qquad(n\ge0,\ 1\le k\le a).
\end{equation}

Conjugating the matrix product in \eqref{eq:matrix-rec} by $W$ changes its
norm only by a fixed factor.  Since $|H_1(t)|\le a+1$, there is a constant
$C_0>0$, independent of $N$, such that
\[
  |H_N(t)|
  \le C_0\,\|B_{N-1}(t)\cdots B_1(t)\|_2.
\]
Lemma~\ref{lem:one-step} gives
\[
  |H_N(t)|
  \le C_0\alpha^{N-1}
  \exp\!\left(-c_1\sum_{1\le n<N}
    \sum_{1\le k<a}
       \dist{\frac{tf(kG_n)}{2\pi}}^2\right).
\]
On the other hand,
$H_N(t)/\alpha^N\to\kappa\Phi(t)$ and $|\Phi(t)|\ge1/2$.  The partial sums in
the exponential are therefore bounded.  By \eqref{eq:small-fixed-t},
\[
  \sum_{n\ge0}\sum_{1\le k<a}f(kG_n)^2<\infty.
\]

To detect the maximal digit, apply Lemma~\ref{lem:two-step} after grouping the product first into the disjoint pairs
\[
  B_{2j+1}(t)B_{2j}(t)\qquad(j\ge1),
\]
and then, in a separate estimate, into the shifted pairs
\[
  B_{2j+2}(t)B_{2j+1}(t)\qquad(j\ge0).
\]
At most two boundary matrices remain in either grouping, and each has norm at
most $\alpha$ by Lemma~\ref{lem:one-step}.  Thus, up to a constant independent
of $N$,
\[
  \norm{B_{N-1}\cdots B_1}_2
  \le \alpha^{N-1}\exp\left(
  -c_2\sum_{\substack{3\le m<N\\ m\ \mathrm{odd}}}
       \dist{\frac{tf(aG_m)}{2\pi}}^2\right)
\]
and, from the shifted grouping,
\[
  \norm{B_{N-1}\cdots B_1}_2
  \le \alpha^{N-1}\exp\left(
  -c_2\sum_{\substack{2\le m<N\\ m\ \mathrm{even}}}
       \dist{\frac{tf(aG_m)}{2\pi}}^2\right).
\]
The same lower bound for $|H_N(t)|/\alpha^N$ shows that both displayed sums
remain bounded as $N\to\infty$.  The finitely many omitted initial levels are
harmless.  Using \eqref{eq:small-fixed-t} once more gives
\[
  \sum_{n\ge0}f(aG_n)^2<\infty.
\]
Together with the non-maximal digit estimate, this is \eqref{eq:S2}.
\end{proof}

It remains to prove the drift condition \eqref{eq:S1}.  At this stage the quadratic series is already known.  It controls every
quadratic remainder in the ratio recurrence.  We may therefore isolate the
linear motion in the Perron direction.  This is the only point at which ratios
are used, and the argument remains inside the interval where $\Phi$ is bounded
away from zero.

Fix from now on a non-zero $t$ with $\abs t\le t_0$.  By
\eqref{eq:H-to-Phi} and \eqref{eq:Phi-nonzero}, there exists $N_0\ge1$ such
that $H_n(t)\ne0$ for every $n\ge N_0$.  Thus
$r_n(t)=H_n(t)/H_{n-1}(t)$ is defined for every $n>N_0$, and
\[
  \frac{r_n(t)}{G_n/G_{n-1}}
  =\frac{H_n(t)/G_n}{H_{n-1}(t)/G_{n-1}}
  \longrightarrow1.
\]
Since $G_n/G_{n-1}\to\alpha$, we obtain
\[
  r_n(t)\longrightarrow\alpha.
\]
For $n>N_0$, put $\varepsilon_n(t)=r_n(t)-\alpha$.

The next argument is the order-two analogue of the ratio perturbation used in
the Zeckendorf proof; compare \cite[Lemma~6]{DrmotaVerwee}.  Because both
$p_n$ and the Perron drift $u_n$ occur here, we record the needed estimate in
our present notation.

\begin{lemma}\label{lem:epsilon}
For this fixed $t$,
\[
  \sum_{n>N_0}\abs{\varepsilon_n(t)}^2<\infty,
  \qquad
  \sum_{n>N_0}\varepsilon_n(t)\quad\text{converges}.
\]
\end{lemma}

\begin{proof}
For $n>N_0$, the ratio recurrence obtained from \eqref{eq:H-rec} is
\[
  r_{n+1}=\sigma_n(a;t)
    +\frac{g_t(aG_{n})\sigma_{n-1}(b;t)}{r_n}.
\]
Using $\alpha^2=a\alpha+b$, this is equivalent to
\begin{equation}\label{eq:epsilon-identity}
  \alpha\varepsilon_{n+1}+(\alpha-a)\varepsilon_n
  =u_n+p_n\varepsilon_n-\varepsilon_n\varepsilon_{n+1}.
\end{equation}
By Proposition~\ref{prop:S2-necessity} and the estimates used in
Lemma~\ref{lem:taylor-u}, the sequences $(u_n)$ and $(p_n)$ belong to
$\ell^2$.
Solving the ratio recurrence for $\varepsilon_{n+1}$ gives
\[
  \varepsilon_{n+1}
  =\frac{u_n-(\alpha-a-p_n)\varepsilon_n}{\alpha+\varepsilon_n}.
\]
Since $p_n\to0$ and $\varepsilon_n\to0$, the coefficient of
$\varepsilon_n$ tends in modulus to
\[
  \frac{\alpha-a}{\alpha}=\frac{b}{\alpha^2}<1.
\]
Thus, for all sufficiently large $n$,
\[
  |\varepsilon_{n+1}|
  \le L|\varepsilon_n|+C|u_n|
\]
with $L<1$.  Iteration and the $\ell^2$ boundedness of convolution by a
geometric kernel yield $(\varepsilon_n)_{n>N_0}\in\ell^2$.

On the other hand,
\[
  \frac{H_n(t)}{\alpha^n}\longrightarrow\kappa\Phi(t)\ne0.
\]
Moreover, for $n>N_0$,
\[
  \prod_{j=N_0+1}^{n}\left(1+\frac{\varepsilon_j(t)}\alpha\right)
  =\frac{\alpha^{N_0}H_n(t)}{H_{N_0}(t)\alpha^n}
  \longrightarrow
  \frac{\alpha^{N_0}\kappa\Phi(t)}{H_{N_0}(t)}\ne0.
\]
Since $\varepsilon_n(t)\to0$ and
$\sum_{n>N_0}|\varepsilon_n(t)|^2<\infty$, choose a tail on which
$|\varepsilon_n(t)/\alpha|<1/2$.  The principal logarithm is then defined for
each factor and
\[
  \log(1+z)=z+O(|z|^2).
\]
The partial products on this tail converge to a non-zero limit, so, after
possibly discarding finitely many further factors, they remain in a simply
connected neighbourhood $U$ of that limit which avoids the origin.  If $L$ is
a branch of the logarithm on $U$, then the partial sum of the principal
logarithms differs from $L$ of the corresponding partial product by an integer
multiple of $2\pi i$.  Since the factors tend to $1$ and the partial products
converge, this integer is eventually constant.  Hence the logarithmic partial
sums converge.  The quadratic remainders are absolutely summable, and therefore
$\sum_{n>N_0}\varepsilon_n(t)$ converges.
\end{proof}

\begin{proposition}
If $f$ has a limiting distribution, then the drift series
\eqref{eq:S1} converges.
\end{proposition}

\begin{proof}
Sum \eqref{eq:epsilon-identity} over $n>N_0$.  By
Lemma~\ref{lem:epsilon}, the two linear $\varepsilon$-series converge, while
\[
  \sum_{n>N_0} p_n\varepsilon_n,
  \qquad
  \sum_{n>N_0}\varepsilon_n\varepsilon_{n+1}
\]
converge absolutely by Cauchy--Schwarz.  Hence
$\sum_{n>N_0}u_n(t)$ converges.
Finally, apply \eqref{eq:u-expansion} with any fixed
$T\ge |t|$.  Together with \eqref{eq:S2}, it gives
\[
  u_n(t)=itD_n+R_n(t),
  \qquad
  \sum_n\abs{R_n(t)}<\infty.
\]
Since $D_n$ is real and $t\ne0$, the tail $\sum_{n>N_0}D_n$ converges.
Adding the finitely many initial terms gives \eqref{eq:S1}.
\end{proof}

\subsection{Characteristic-function representations}

The equivalence in Theorem~\ref{thm:main} is now proved: the preceding
subsections establish both the sufficiency and the necessity of
\eqref{eq:S1} and \eqref{eq:S2}.  It remains to prove the two representations
of the limiting characteristic function.  The matrix expression is meaningful
at every real frequency, whereas the scalar infinite product is obtained on a
neighbourhood of the origin, where all finite characteristic sums can be kept
non-zero.

Assume that \eqref{eq:S1} and \eqref{eq:S2} hold.  By
\eqref{eq:natural-cutoff-limit}, $H_n(t)/G_n$ converges locally uniformly to
$\Phi(t)$, and $\Phi(0)=1$.  Hence there are $t_1>0$ and $N\ge1$ such that
$H_n(t)\ne0$ for every $n\ge N$ and $\abs t\le t_1$.  For each of the finitely
many indices $0\le n<N$, one has $H_n(0)=G_n>0$; by continuity, after decreasing
$t_1$ if necessary, $H_n(t)\ne0$ also for these indices.  Thus all ratios
$r_n(t)=H_n(t)/H_{n-1}(t)$ are defined for $n\ge1$ and $\abs t\le t_1$.
Since $H_0(t)=1$, telescoping gives
\[
  \prod_{j=1}^n\frac{r_j(t)}{\alpha}
  =\frac{H_n(t)}{\alpha^n}.
\]
Using $G_n/\alpha^n\to\kappa$ together with
\eqref{eq:natural-cutoff-limit} yields \eqref{eq:ratio-product-Phi} with
$t_0=t_1$.

Finally, iterating \eqref{eq:matrix-rec} gives
\[
  H_n(t)
  =\begin{pmatrix}1&0\end{pmatrix}
   A_{n-1}(t)\cdots A_1(t)
   \binom{H_1(t)}{1}.
\]
Combining this identity with $H_n(t)/G_n\to\Phi(t)$ and
$G_n/\alpha^n\to\kappa$ proves the global formula
\eqref{eq:matrix-product-Phi}.

\hypertarget{sec:examples}{}
\section{Examples}

The theorem contains the elementary absolutely summable regime, but the drift
condition is genuinely weaker than absolute convergence.  The next example
exhibits the distinction in the simplest possible form.

\begin{example}
Let
\[
  f(aG_n)=\frac{(-1)^n}{b(n+1)},
  \qquad
  f(kG_n)=0\quad(1\le k<a).
\]
Then the drift series is a tail of the alternating harmonic series, whereas the
quadratic series is dominated by $\sum_n(n+1)^{-2}$.  Thus the limiting
distribution exists although the digit-level values are not absolutely
summable.  This example lies beyond the elementary absolutely summable regime
\(\sum_{n,k}\abs{f(kG_n)}<\infty\).
\end{example}

\begin{example}
Assume that
\[
  \sum_{n\ge0}\sum_{1\le k\le a}\abs{f(kG_n)}<\infty.
\]
Then both \eqref{eq:S1} and \eqref{eq:S2} converge absolutely, so $f$ has a
limiting distribution.  This elementary regime is useful as a reference
point, but it does not account for the conditional cancellation allowed in the
preceding example.
\end{example}

\hypertarget{sec:special-cases}{}
\section{Special cases}

The general drift simplifies in several familiar subfamilies.

\begin{corollary}
Let
\[
  G_{n+2}=aG_{n+1}+G_n.
\]
Then a real-valued $G$-additive function has a limiting distribution if and only
if
\[
  \sum_{n\ge1}\left(
    \alpha\sum_{1\le k<a}f(kG_n)+f(aG_n)
  \right)
\]
and
\[
  \sum_{n\ge0}\sum_{1\le k\le a}f(kG_n)^2
\]
converge.
\end{corollary}

\begin{proof}
This is Theorem~\ref{thm:main} with $b=1$; the middle sum in
\eqref{eq:Dn} is empty.
\end{proof}

\begin{corollary}
If $b=a$, the pair of conditions \eqref{eq:S1} and \eqref{eq:S2} is equivalent
to \eqref{eq:S2} together with the convergence of
\[
  \sum_{n\ge1}\left(
    (\alpha+1)\sum_{1\le k<a}f(kG_n)+a f(aG_n)
  \right).
\]
\end{corollary}

\begin{proof}
Put
\[
  A_n=\sum_{1\le k<a}f(kG_n).
\]
Under \eqref{eq:S2}, one has $A_n\to0$.  When $b=a$, the $n$-th term of
\eqref{eq:S1} is
\[
  D_n=\alpha A_n+A_{n-1}+a f(aG_n),
\]
whereas the $n$-th term of the displayed series in the statement is
\[
  E_n=(\alpha+1)A_n+a f(aG_n).
\]
Hence
\[
  \sum_{n=1}^N(D_n-E_n)=A_0-A_N\longrightarrow A_0.
\]
Thus, under \eqref{eq:S2}, the two drift series converge simultaneously.
\end{proof}

\begin{corollary}
For $a=b=1$, the criterion reduces to
\[
  \sum_{n\ge0}f(G_n)\quad\text{and}\quad
  \sum_{n\ge0}f(G_n)^2,
\]
which is exactly the Erd\H{o}s--Wintner criterion for Zeckendorf-additive
functions proved in \cite[Theorem~7]{DrmotaVerwee}.
\end{corollary}

\begin{proof}
Let $(F_n)$ be the Fibonacci sequence with $F_0=0$ and $F_1=1$.  Here
$G_n=F_{n+2}$, and \eqref{eq:Dn} reduces to $D_n=f(G_n)$.  Shifting the
initial index changes only finitely many terms.
\end{proof}

Theorem~\ref{thm:main} also gives two immediate stability consequences. If
$f$ and $g$ are real-valued $G$-additive functions with limiting distributions,
then $f+g$ has a limiting distribution: the drift condition is linear, while
\((x+y)^2\le2x^2+2y^2\) preserves the quadratic condition.

In particular, if $f$ has a limiting distribution and $g$ is $G$-additive with
\[
  \sum_{n\ge0}\sum_{1\le k\le a}\abs{g(kG_n)}<\infty,
\]
then $f+g$ has a limiting distribution, since this assumption implies the
quadratic condition for $g$ and absolute convergence of its drift series.

The limiting law of $f+g$ need not be the convolution of the two individual
laws, because $f(n)$ and $g(n)$ are evaluated on the same digit expansion and
are generally dependent.

\hypertarget{sec:further-directions}{}
\section{Further directions}

We begin with the natural higher-order conjecture. Let $d\ge2$ and let
$a_0,\ldots,a_{d-1}$ be positive integers. Consider
\[
  G_{n+d}=a_0G_{n+d-1}+\cdots+a_{d-1}G_n\qquad(n\ge0),
\]
with $G_0=1$ and
\[
  G_r=a_0G_{r-1}+\cdots+a_{r-1}G_0+1
  \qquad(0<r<d).
\]
We call $(G_n)$ a linear recurrent base if every non-negative integer has a
unique greedy $G$-expansion, the companion matrix is primitive, and its
dominant root $\alpha>1$ is Pisot with $G_n/\alpha^n\to\kappa>0$. Put
\[
  c=\max_{0\le r<d}a_r.
\]
For a real-valued $G$-additive function $f$, define
\[
  \mathcal D_n
  =\sum_{j<d}\frac1{\alpha^j}
    \sum_{0\le k<a_j}
    \left(
      f(kG_{n+d-j})
      +\sum_{\ell<j}f(a_\ell G_{n+d-\ell})
    \right).
\]

\begin{conjecture}\label{conj:general-lrb}
For every linear recurrent base $(G_n)$ and every real-valued $G$-additive
function $f$, the function $f$ has a limiting distribution if and only if
\[
  \sum_{n\ge0}\mathcal D_n
\]
converges and
\[
  \sum_{n\ge0}\sum_{1\le k\le c}f(kG_n)^2<\infty.
\]
\end{conjecture}

Theorem~\ref{thm:main} proves Conjecture~\ref{conj:general-lrb} in order two.
Indeed, when $d=2$, $a_0=a$ and $a_1=b$, one has, for $n\ge2$,
\[
  \alpha\mathcal D_{n-2}
  =\alpha\sum_{0\le k<a}f(kG_n)
   +\sum_{0\le k<b}f(kG_{n-1})
   +b f(aG_n)
  =D_n,
\]
since $f(0)=0$. Thus the two drift conditions agree up to a non-zero constant
factor and finitely many initial terms, while $c=a$ because $b\le a$.

The order-two proof relies on the explicit Perron symmetrization and on the
fact that a block of length two detects the only digit value invisible to a
one-step norm. In higher order, several digit coordinates may be invisible to
a single matrix, and a fixed block would have to detect them simultaneously at
the Perron scale. Constructing such a coercive block estimate appears to be
the main additional step in this approach.

There is a complementary route in which the finite-state structure of the
digit language is taken as the primary object. Numeration systems defined by
regular languages and their additive functions have been studied by Grabner
and Rigo \cite{GrabnerRigo2003,GrabnerRigo2007}; for the symbolic-dynamical
background on sofic shifts and finite-state codes, see
\cite[Chapters~3 and~5]{LindMarcus}. When the admissible language is regular,
a bounded-memory lifting produces a finite directed graph on which the local
digit contributions become position-dependent edge observables. The
Perron--Frobenius eigenvectors determine the associated Parry chain, and the
covariance form of this chain gives the natural Green energy. On a
finite-dimensional coordinate space whose covariance spectral density is
uniformly positive, this energy is comparable with the ordinary Euclidean
square sum. In standard Parry-type presentations of linear recurrent bases,
zero-padding admissible words rules out hidden twisted coboundary directions.
This suggests that Conjecture~\ref{conj:general-lrb} may ultimately be
approached through spectral positivity rather than through explicit norm
contractions for higher-dimensional companion matrices.

The obstruction in a general finite-state system is precisely the kernel of
this Green energy. If an edge observable has the form
$F(e)=\Psi(b(e))-\Psi(a(e))$, its partial sums telescope, so a one-site square
condition cannot be necessary in that direction. More generally, at non-zero
frequencies the Green kernel may contain twisted coboundary directions; the
quotient language below always refers to the full Green kernel, not only to
ordinary coboundaries. We shall call a sequence \emph{Green--Cauchy} when the
Green energy of every tail $[m,n)$ tends to zero as $m,n\to\infty$; coboundary
components are retained as boundary terms and must have compatible limiting
behaviour. Finite transducers fit the same picture after passing to a reachable
product or context graph; see, for instance, the recent numeration construction
of Charlier, Popoli and Rigo \cite{CharlierPopoliRigo}. The primitive case then
falls back into the fixed finite-state framework, while reducible or periodic
product graphs lead to the first genuinely broader problem.

\begin{conjecture}\label{conj:transducer-EW}
Let $\mathcal T$ be a finite transducer driven by a topologically mixing sofic
input system equipped with its measure of maximal entropy, and fix an initial
transducer state. Let $(g_n)_{n\ge0}$ be real-valued position-dependent output
observables of uniformly bounded memory such that $\|g_n\|_\infty\to0$.
Decompose the reachable product graph into its recurrent components and cyclic
classes. On each recurrent component, centre the output observable and separate
its class modulo the Green kernel from its finite-state coboundary part.

Then the additive output sums converge in distribution if and only if, on
every recurrent component reached with positive asymptotic weight, the drift
series converges, the quotient component is Green--Cauchy, and the coboundary
boundary term has compatible limiting behaviour, with the resulting weighted
mixture of component laws independent of the length modulo a common period.
\end{conjecture}

Conjecture~\ref{conj:transducer-EW} still depends on a transducer realization.
The maximal formulation should instead be intrinsic to the underlying sofic
system. We state it in the infinitesimal regime in order not to introduce a
spurious large-jump obstruction along telescoping coboundary directions; a
non-infinitesimal version should impose the corresponding large-jump condition
after the same Green--coboundary reduction.

\begin{conjecture}\label{conj:sofic-EW}
Let $X$ be a topologically mixing sofic system with measure of maximal entropy
$\mu$, and let $(g_n)_{n\ge0}$ be a sequence of real-valued local observables
of uniformly bounded memory such that $\|g_n\|_\infty\to0$. Put
\[
  d_n=\int_X g_n\,d\mu,
  \qquad
  \widetilde g_n=g_n-d_n.
\]
Choose a primitive right-resolving finite presentation of $X$, pass if
necessary to a finite context graph, and lift $(\widetilde g_n)$ to edge
observables. After separating the finite-state coboundary component, the laws
under $\mu$ of
\[
  S_N=\sum_{n<N}g_n
\]
converge if and only if the following conditions hold:
\begin{enumerate}[label=\textup{(\roman*)}]
\item the drift series
\[
  \sum_{n\ge0}d_n
\]
converges;
\item the image of $(\widetilde g_n)$ in the quotient by the finite-state
Green kernel is Green--Cauchy;
\item the induced coboundary boundary term is tight and has the joint limiting
behaviour required to combine with the Green--Cauchy component.
\end{enumerate}
Moreover, these conditions and the resulting limiting law are independent of
the chosen right-resolving presentation and of the finite-memory lifting.
\end{conjecture}

When the lifted coordinate space has uniformly positive Green density, the
quotient condition in Conjecture~\ref{conj:sofic-EW} reduces to ordinary square
summability. The three conjectures therefore form a natural hierarchy: the
explicit linear-recurrent criterion, its transducer extension, and finally an
intrinsic finite-state Erd\H{o}s--Wintner principle for sofic systems.

\appendix
\hypertarget{app:perron-product}{}
\section{The second-order Perron product lemma}

We prove here Proposition~\ref{prop:product}.  Thus, let
$A_n=A+E_n$, where $A$ has a simple dominant eigenvalue $\alpha$ and a
second eigenvalue $\lambda$ with $|\lambda|<\alpha$.  Under the assumptions
\[
  \sum_{n\ge0}\|E_n\|^2<\infty
  \qquad\text{and}\qquad
  \sum_{n\ge0}\ell^{\mathsf T}E_n r
  \quad\text{converges},
\]
our goal is to prove that
\[
  \alpha^{-N}A_{N-1}\cdots A_0
\]
converges as $N\to\infty$, together with the locally uniform version when
the matrices depend on a parameter.

General asymptotic-integration results for perturbed linear difference
systems may be found in the work of Coffman and of Benzaid--Lutz
\cite{Coffman,BenzaidLutz}.  Here we give a direct proof adapted to the
particular combination needed in Subsection~\ref{sec:suff}: the full
perturbation is square-summable, while the drift along the Perron direction
is only assumed to converge, not necessarily absolutely.  This also isolates
the point at which the argument goes beyond the absolute coefficient
perturbation used in \cite[Lemma~2]{BaratGrabner}.

\begin{proof}
We separate the Perron direction from the contracting eigendirection.
After diagonalizing $A$ and dividing by its Perron root $\alpha$, the Perron
coordinate has unperturbed multiplier $1$, whereas the second coordinate has
multiplier $\vartheta$ with $|\vartheta|<1$.  We straighten the small coupling
between these two directions and reduce the problem to a triangular product.

\medskip
\noindent\emph{Step 1: diagonalize the main term.}
Choose a right eigenvector $s$ of $A$ associated with $\lambda$ and put
$S=(r,s)$.  Since $A$ is diagonalizable,
\[
  S^{-1}AS=
  \begin{pmatrix}
    \alpha&0\\
    0&\lambda
  \end{pmatrix}.
\]
For each $n$, define
\[
  M_n:=\alpha^{-1}S^{-1}A_nS.
\]
Then
\[
  M_{N-1}\cdots M_0
  =\alpha^{-N}S^{-1}A_{N-1}\cdots A_0S.
\]
Hence the normalized products in Proposition~\ref{prop:product} converge if
and only if the products $M_{N-1}\cdots M_0$ converge.

Write
\[
  \alpha^{-1}S^{-1}E_nS
  =
  \begin{pmatrix}
    \xi_n&\beta_n\\
    \gamma_n&\zeta_n
  \end{pmatrix}.
\]
Since $A_n=A+E_n$,
\[
  M_n=
  \begin{pmatrix}
    1+\xi_n&\beta_n\\
    \gamma_n&\vartheta+\zeta_n
  \end{pmatrix},
  \qquad
  \vartheta:=\frac{\lambda}{\alpha},
  \qquad |\vartheta|<1.
\]
The map $E\mapsto \alpha^{-1}S^{-1}ES$ is bounded on
$\C^{2\times2}$, so the hypothesis $\sum_n\|E_n\|^2<\infty$ gives
\[
  (\xi_n),\;(\beta_n),\;(\gamma_n),\;(\zeta_n)\in\ell^2.
\]

We next identify the perturbation in the Perron coordinate.  Since
$\ell^{\mathsf T}A=\alpha\ell^{\mathsf T}$ and $As=\lambda s$,
\[
  (\alpha-\lambda)\ell^{\mathsf T}s=0.
\]
Thus $\ell^{\mathsf T}s=0$, and the first row of $S^{-1}$ is
$\ell^{\mathsf T}/(\ell^{\mathsf T}r)$.  Consequently,
\[
  \xi_n
  =\frac{\ell^{\mathsf T}E_n r}
         {\alpha\,\ell^{\mathsf T}r}.
\]
The hypothesis on the Perron drift therefore gives convergence of
$\sum_n\xi_n$.  In these coordinates, $\xi_n$ is the first-order perturbation
along the Perron direction, $\beta_n$ and $\gamma_n$ are the two off-diagonal
couplings, and $\zeta_n$ perturbs the contracting direction.

\medskip
\noindent\emph{Step 2: straighten the perturbed Perron direction.}
It is enough to control a tail of the product.  Choose
$|\vartheta|<\vartheta_0<\rho<1$, and then choose $\eta>0$ so small that
\[
  \frac{\vartheta_0+\eta}{1-2\eta}\le\rho
  \qquad\text{and}\qquad
  \frac{\eta}{1-2\eta}+\rho\le1.
\]
Since the four perturbation entries tend to zero, we may fix $n_0$ so that
all of them have modulus at most $\eta$ for $n\ge n_0$.

We seek numbers $z_n$ and $\mu_n$ such that
\[
  M_n\binom{1}{z_n}
  =\mu_n\binom{1}{z_{n+1}}.
\]
Starting from $z_{n_0}=0$, the first coordinate gives
\[
  \mu_n=1+\xi_n+\beta_n z_n,
\]
and the second gives the Riccati recurrence
\begin{equation}\label{eq:Riccati}
  z_{n+1}
  =\frac{\gamma_n+(\vartheta+\zeta_n)z_n}
         {1+\xi_n+\beta_nz_n}.
\end{equation}
If $|z_n|\le1$, then
\[
  |1+\xi_n+\beta_nz_n|\ge1-2\eta
\]
and hence
\[
  |z_{n+1}|
  \le \frac{|\gamma_n|}{1-2\eta}
     +\frac{\vartheta_0+\eta}{1-2\eta}|z_n|
  \le \frac{|\gamma_n|}{1-2\eta}+\rho|z_n|.
\]
Our choice of $\eta$ shows inductively that $|z_n|\le1$ for every
$n\ge n_0$.  With $C=(1-2\eta)^{-1}$, iteration gives
\[
  |z_n|
  \le C\sum_{n_0\le j<n}\rho^{n-1-j}|\gamma_j|.
\]
Young's convolution inequality gives $(z_n)\in\ell^2$; in particular,
$z_n\to0$.

Now put
\[
  T_n=
  \begin{pmatrix}
    1&0\\
    z_n&1
  \end{pmatrix}.
\]
By \eqref{eq:Riccati},
\[
  T_{n+1}^{-1}M_nT_n
  =U_n:=
  \begin{pmatrix}
    \mu_n&\beta_n\\
    0&\nu_n
  \end{pmatrix},
\]
where
\[
  \mu_n=1+\xi_n+\beta_nz_n,
  \qquad
  \nu_n=\vartheta+\zeta_n-z_{n+1}\beta_n.
\]
Thus the moving change of coordinates removes the lower-left coupling and
reduces the problem to a triangular product.

\medskip
\noindent\emph{Step 3: convergence of the triangular product.}
Since $(\beta_n)_{n\ge n_0}$ and $(z_n)_{n\ge n_0}$ belong to $\ell^2$,
Cauchy--Schwarz gives
\[
  \sum_{n\ge n_0}|\beta_nz_n|<\infty.
\]
Since
\(
  \mu_n-1=\xi_n+\beta_nz_n,
\)
the series $\sum_{n\ge n_0}(\mu_n-1)$ converges and
$(\mu_n-1)\in\ell^2$.  Moreover,
$z_n\to0$ and $\zeta_n,\beta_n\to0$, whence
\[
  \nu_n\longrightarrow\vartheta.
\]

Choose $n_1\ge n_0$ and $\rho_1<1$ so that, for every $n\ge n_1$,
\[
  |\mu_n-1|\le\frac12,
  \qquad
  |\nu_n|\le\rho_1.
\]
For $|w|\le1/2$,
\[
  \log(1+w)=w+O(|w|^2),
\]
with the principal branch of the logarithm.  Therefore
$\sum_{n\ge n_1}\log\mu_n$ converges, and
\[
  P_N:=\prod_{n_1\le k<N}\mu_k
\]
converges to a non-zero limit.  In particular, $(P_N)$ and $(P_N^{-1})$ are
bounded.

For $N>n_1$, let $B_N$ denote the upper-right entry of
$U_{N-1}\cdots U_{n_1}$.  Factoring out the upper-left product gives
\[
  B_N
  =P_N\sum_{j=n_1}^{N-1}
    \frac{\beta_j}{P_{j+1}}
    \prod_{n_1\le k<j}\nu_k.
\]
The sequences $(P_N)$ and $(P_N^{-1})$ are bounded, and therefore the
$j$-th summand is bounded by
\[
  C|\beta_j|\rho_1^{j-n_1}.
\]
This majorant is summable by Cauchy--Schwarz.  Hence the series on the
right converges absolutely as $N\to\infty$.  Since $(P_N)$ also converges,
$(B_N)$ converges.  The upper-left entry is $P_N$,
the lower-left entry is zero, and the lower-right entry is
$\prod_{n_1\le k<N}\nu_k$, which tends to zero.  Thus
\[
  U_{N-1}\cdots U_{n_1}
\]
converges.  Multiplying on the right by the fixed finite product
$U_{n_1-1}\cdots U_{n_0}$ shows that
$U_{N-1}\cdots U_{n_0}$ converges as well.

Finally,
\[
  M_{N-1}\cdots M_{n_0}
  =T_NU_{N-1}\cdots U_{n_0}T_{n_0}^{-1},
\]
and $T_N\to I$.  Hence the tail products of the $M_n$ converge.  Reintroducing
the finitely many initial factors and conjugating back by $S$ proves the first
assertion of Proposition~\ref{prop:product}.

\medskip
\noindent\emph{Uniformity in the parameter.}
Suppose now that the matrices depend on $t$ as in the second part of the
proposition, and fix $R>0$.  For $|t|\le R$, the hypotheses imply
\[
  \sum_n\sup_{|t|\le R}
  \bigl(|\xi_n(t)|^2+|\beta_n(t)|^2
        +|\gamma_n(t)|^2+|\zeta_n(t)|^2\bigr)<\infty,
\]
and $\sum_n\xi_n(t)$ converges uniformly.  The index $n_0$ in Step~2 can
therefore be chosen independently of $t$.  Since the denominator in
\eqref{eq:Riccati} is then bounded below by $1-2\eta$, the recurrence also
shows inductively that each $z_n(t)$, and hence each $U_n(t)$, is continuous
on $|t|\le R$.

Set
\[
  c_n:=\sup_{|t|\le R}|\gamma_n(t)|,
  \qquad
  w_n:=\sup_{|t|\le R}|z_n(t)|.
\]
The estimate from Step~2 gives
\[
  w_n\le C\sum_{n_0\le j<n}\rho^{n-1-j}c_j.
\]
Since $(c_n)\in\ell^2$, Young's inequality yields $(w_n)\in\ell^2$; in
particular, $w_n\to0$.

Similarly, put
\[
  b_n:=\sup_{|t|\le R}|\beta_n(t)|.
\]
Then $(b_n)\in\ell^2$, and
\[
  \sum_n\sup_{|t|\le R}|\beta_n(t)z_n(t)|
  \le\sum_n b_nw_n<\infty.
\]
Thus $\sum_n(\mu_n(t)-1)$ converges uniformly on $|t|\le R$.  Moreover,
\[
  \sum_n\sup_{|t|\le R}|\mu_n(t)-1|^2<\infty.
\]
The relations above also give $\nu_n(t)\to\vartheta$ uniformly on
$|t|\le R$.

We may therefore choose $n_1\ge n_0$ and $\rho_1<1$, independently of $t$,
such that
\[
  |\mu_n(t)-1|\le\frac12,
  \qquad
  |\nu_n(t)|\le\rho_1
  \qquad(n\ge n_1,\ |t|\le R).
\]
The logarithmic expansion is uniform on this disc.  Hence
\[
  P_N(t):=\prod_{n_1\le k<N}\mu_k(t)
\]
converges uniformly to a function bounded away from zero on $|t|\le R$.

Let $B_N(t)$ be the upper-right entry of
$U_{N-1}(t)\cdots U_{n_1}(t)$.  As above,
\[
  B_N(t)
  =P_N(t)\sum_{j=n_1}^{N-1}
    \frac{\beta_j(t)}{P_{j+1}(t)}
    \prod_{n_1\le k<j}\nu_k(t).
\]
Uniform convergence of $P_N(t)$ to a function bounded away from zero implies
that $P_N(t)$ and $P_N(t)^{-1}$ are uniformly bounded for
$N\ge n_1$ and $|t|\le R$.  The $j$-th summand is therefore bounded by
\[
  Cb_j\rho_1^{j-n_1}.
\]
This sequence is summable, so the Weierstrass $M$-test gives uniform
convergence of the series.  Together with the uniform convergence of $P_N(t)$,
this proves uniform convergence of $B_N(t)$.  The remaining three entries
converge uniformly as well, and hence
$U_{N-1}(t)\cdots U_{n_1}(t)$ converges uniformly on $|t|\le R$.
The finite block $U_{n_1-1}(t)\cdots U_{n_0}(t)$ is continuous, hence
uniformly bounded on the compact interval, so multiplication by it preserves
uniform convergence.

Finally, $w_N\to0$ gives $T_N(t)\to I$ uniformly.  The identity
\[
  M_{N-1}(t)\cdots M_{n_0}(t)
  =T_N(t)U_{N-1}(t)\cdots U_{n_0}(t)T_{n_0}^{-1}
\]
therefore yields uniform convergence of the tail product.  Multiplication on
the right by the finite continuous product
$M_{n_0-1}(t)\cdots M_0(t)$ gives uniform convergence of
$M_{N-1}(t)\cdots M_0(t)$ on $|t|\le R$.  Since each finite product is
continuous in $t$, its uniform limit is continuous.  As $R$ was arbitrary,
the convergence is locally uniform.
\end{proof}


\begin{thebibliography}{99}

\bibitem{BenzaidLutz}
Z.~Benzaid and D.~A. Lutz,
\emph{Asymptotic representation of solutions of perturbed systems of linear
  difference equations},
Stud. Appl. Math. \textbf{77} (1987), 195--221.

\bibitem{BaratGrabner}
G.~Barat and P.~J. Grabner,
\emph{Distribution properties of $G$-additive functions},
J. Number Theory \textbf{60} (1996), 103--123.

\bibitem{CharlierPopoliRigo}
E.~Charlier, P.~Popoli and M.~Rigo,
\emph{Computing expansions in infinitely many Cantor real bases via a single
transducer},
arXiv:2507.04848, 2025.

\bibitem{Coffman}
C.~V. Coffman,
\emph{Asymptotic behavior of solutions of ordinary difference equations},
Trans. Amer. Math. Soc. \textbf{110} (1964), 22--51.

\bibitem{Delange}
H.~Delange,
\emph{Sur les fonctions $q$-additives ou $q$-multiplicatives},
Acta Arith. \textbf{21} (1972), 285--298.

\bibitem{DrmotaVerwee}
M.~Drmota and J.~Verwee,
\emph{Effective Erd\H{o}s--Wintner theorems for digital expansions},
J. Number Theory \textbf{229} (2021), 218--260.

\bibitem{ErdosWintner}
P.~Erd\H{o}s and A.~Wintner,
\emph{Additive arithmetical functions and statistical independence},
Amer. J. Math. \textbf{61} (1939), 713--721.

\bibitem{GrabnerRigo2003}
P.~J. Grabner and M.~Rigo,
\emph{Additive functions with respect to numeration systems on regular
languages},
Monatsh. Math. \textbf{139} (2003), 205--219.

\bibitem{GrabnerRigo2007}
P.~J. Grabner and M.~Rigo,
\emph{Distribution of additive functions with respect to numeration systems
on regular languages},
Theory Comput. Syst. \textbf{40} (2007), 205--223.

\bibitem{GrabnerTichy}
P.~J. Grabner and R.~F. Tichy,
\emph{Contributions to digit expansions with respect to linear recurrences},
J. Number Theory \textbf{36} (1990), 160--169.

\bibitem{LindMarcus}
D.~Lind and B.~Marcus,
\emph{An Introduction to Symbolic Dynamics and Coding},
2nd ed., Cambridge University Press, Cambridge, 2021.

\bibitem{TenenbaumVerwee}
G.~Tenenbaum and J.~Verwee,
\emph{Effective Erd\H{o}s--Wintner theorems},
Proc. Steklov Inst. Math. \textbf{314} (2021), 264--278.

\bibitem{VerweeLRBWithdrawn}
J.~Verwee,
\emph{Erd\H{o}s--Wintner theorem for linear recurrent bases},
arXiv:2512.20882, withdrawn, 2026.

\bibitem{VerweeCantor}
J.~Verwee,
\emph{Effective Erd\H{o}s--Wintner for Cantor numeration systems via a
trailing-window method},
Int. J. Number Theory \textbf{22} (2026), 1175--1194.

\end{thebibliography}
\end{document}